%% file: main.tex
\documentclass[journal,twoside,web]{ieeecolor}
\usepackage{lcsys}
\usepackage{cite}
\usepackage{amsmath,amssymb,amsfonts}
\usepackage{algorithmic}
\usepackage{graphicx}
\usepackage{textcomp}

\usepackage{mathtools}
\usepackage{commath}
\usepackage{subcaption}
\usepackage{nicefrac}       
\usepackage{bm}
\usepackage{bbm}

\newtheorem{definition}{Definition}[section]
\newtheorem{assumption}{Assumption}[section]
\newtheorem{lemma}{Lemma}[section]
\newtheorem{theorem}{Theorem}[section]
\newtheorem{proposition}{Proposition}[section]
\newtheorem{remark}{Remark}

\begin{document}
\title{Convergence Analysis of EXTRA in Non-convex Distributed Optimization}
\author{Lei Qin and Ye Pu
\thanks{This work was supported by a Melbourne Research Scholarship and the Australian Research Council (DE220101527).}
\thanks{L. Qin and Y. Pu are with the Department of Electrical and Electronic Engineering, University of Melbourne, Parkville VIC 3010, Australia \texttt{\small leqin@student.unimelb.edu.au, ye.pu@unimelb.edu.au}.}}
\maketitle
\thispagestyle{empty} 

\begin{abstract}
Optimization problems involving the minimization of a finite sum of smooth, possibly non-convex functions arise in numerous applications. To achieve a consensus solution over a network, distributed optimization algorithms, such as \textbf{EXTRA} (decentralized exact first-order algorithm), have been proposed to address these challenges. In this paper, we analyze the convergence properties of \textbf{EXTRA} in the context of smooth, non-convex optimization. By interpreting its updates as a nonlinear dynamical system, we show novel insights into its convergence properties. Specifically, i) \textbf{EXTRA} converges to a consensual first-order stationary point of the global objective with a sublinear rate; and ii) \textbf{EXTRA} almost surely avoids consensual strict saddle points under random initialization, providing robust second-order guarantees. These findings provide a deeper understanding of \textbf{EXTRA} in a non-convex context.
\end{abstract}


\section{Introduction}
\label{sec: Introduction}
\input{Sections/Introduction.tex}

\section{Assumptions and Supporting Results}
\label{sec: Assumptions and Supporting Results}
\input{Sections/Assumptions_and_Supporting_Results.tex}

\section{Main Results}
\label{sec: Main Results}
\input{Sections/Main_Results.tex}

\section{Proofs}
\label{sec: Proofs}
\input{Sections/Proofs}

\section{Numerical Examples}
\label{sec: Numerical Examples}
\input{Sections/Numerical_Examples}

\section{Conclusions}
\label{sec: Conclusions and Discussion}
\input{Sections/Conclusions_and_Discussion}



\bibliographystyle{unsrt}  
\bibliography{Reference} 





\end{document}

%% file: Sections/Introduction.tex
\IEEEPARstart{O}{ptimization} problems that involve a finite sum of functions formed as
\begin{gather}
    \label{eq: Unconstrained optimization problem}
    \min_{\mathbf{x} \in \mathbb{R}^{n}} f(\mathbf{x}), \quad \text{where } f(\mathbf{x}) \triangleq \sum_{i=1}^{m} f_{i}(\mathbf{x}),
\end{gather}
are central to a wide range of applications, including distributed machine learning \cite{boyd2011distributed,konevcny2015federated,mcmahan2017communication}, sensor networks\cite{rabbat2004distributed,johansson2007simple,wan2009event}, and large-scale control \cite{nedic2018distributed,molzahn2017survey,camponogara2010distributed}. In such problems, the function $f_{i}: \mathbb{R}^{n} \rightarrow \mathbb{R}$ is often assumed to be smooth and possibly non-convex. A key challenge arises when $f_{i}$ represents data or tasks that are distributed across $m$ agents, such that each agent $i$ has access only to its local function $f_{i}$ and its gradient information. The agents must collaboratively minimize the global objective $f(x)$ while communicating over an undirected and connected network graph $\mathcal{G}(\mathcal{V},\mathcal{E})$. This graph consists of two key components: the set of agents $\mathcal{V}:=\{1,\ldots,m\}$, where each agent $i \in \mathcal{V}$ processes its local data, and the set of edges $\mathcal{E} \subseteq \mathcal{V} \times \mathcal{V}$, indicating that agents $i$ and $j$ can exchange information if $(i,j)\in\mathcal{E}$.

To address such optimization problems, distributed algorithms have been developed to enable agents to iteratively reach a consensus on the optimal solution. Among these algorithms, consensus-based gradient descent methods have gained significant attention due to their scalability and adaptability to network settings. Notable methods include \textbf{D}istributed \textbf{G}radient \textbf{D}escent (\textbf{DGD}) \cite{nedic2009distributed,nedic2010constrained}, and its variants \cite{bianchi2012convergence,jakovetic2014fast,chen2012fast1,chen2012fast2,duchi2011dual,qin2023second}. However, these methods often face limitations in terms of achieving exact consensus. A key breakthrough in distributed optimization was the introduction of the \textbf{Ex}act Firs\textbf{t}-Orde\textbf{r} \textbf{A}lgorithm (\textbf{EXTRA}) in \cite{shi2015extra} and \textbf{DIGing} \cite{nedic2017achieving}, which both improve on traditional consensus-based methods by achieving exact consensus. As \textbf{EXTRA} in \cite{shi2015extra}, the update for each agent $i \in \mathcal{V}$ in iteration $k \in \mathbb{N}$ is given by
\begin{gather}
\label{eq: EXTRA}
    \begin{gathered}
        \hat{\mathbf{x}}^{1}_{i} = \sum_{j=1}^{m} \mathbf{W}_{ij}\hat{\mathbf{x}}^{0}_j - \alpha\nabla f_{i}(\hat{\mathbf{x}}^{0}_{i}),\\
        \hat{\mathbf{x}}^{k+2}_{i} = \hat{\mathbf{x}}^{k+1}_{i} + \sum_{j=1}^{m} \mathbf{W}_{ij}\hat{\mathbf{x}}^{k+1}_{j} - \sum_{j=1}^{m} \mathbf{V}_{ij}\hat{\mathbf{x}}^{k}_{j}\\
        - \alpha[\nabla f_{i}(\hat{\mathbf{x}}^{k+1}_{i}) - \nabla f_{i}(\hat{\mathbf{x}}^{k}_{i})],  
    \end{gathered}
\end{gather}
where $\hat{\mathbf{x}}_{i} \in \mathbb{R}^{n}$ is the local copy of the decision vector $\mathbf{x}$ at agent $i \in \mathcal{V}$, and $\hat{\mathbf{x}} = [\hat{\mathbf{x}}_{1}^{\top},\cdots, \hat{\mathbf{x}}_{m}^{\top}]^{\top} \in (\mathbb{R}^{n})^{m}$, $\alpha > 0$ is the constant step-size, $\nabla f_{i}$ is the gradient of $f_{i}$, and $\mathbf{W}_{ij}$ (resp., $\mathbf{V}_{ij}$) is the scalar entry in the $i$-th row and $j$-th column of a graph-dependent matrix $\mathbf{W} \in \mathbb{R}^{m\times m}$ (resp., $\mathbf{V} \in \mathbb{R}^{m\times m}$). Compared to gradient tracking methods in \cite{nedic2017achieving}, \textbf{EXTRA} is more free to choose graph-dependent matrices $\mathbf{W}$ and $\mathbf{V}$. 

Despite its effectiveness, the theoretical analysis of the \textbf{EXTRA} algorithm has largely focused on settings where the objective function $f(x)$ is smooth and either convex or strongly convex. The original work on \textbf{EXTRA} \cite{shi2015extra} established linear convergence rates for strongly convex objectives and sublinear rates for general convex functions, providing a solid foundation for distributed optimization under these assumptions. Subsequently, the primal-dual-based method \cite{lei2016primal}, which extends \textbf{EXTRA}, achieved similar convergence results in convex settings. Additionally, convergence results for \textbf{EXTRA} on directed graphs were established in \cite{xin2018linear}, while extensions such as \textbf{PG-EXTRA} \cite{shi2015proximal} have addressed non-smooth convex problems. However, in non-convex settings, theoretical results remain more limited. Recent studies, such as \cite{hong2018gradient,daneshmand2020second}, have started to explore the performance of primal-dual-based methods and gradient tracking methods in non-convex settings, focusing on second-order guarantees.

This paper provides a novel perspective on the \textbf{EXTRA} algorithm by analyzing its update as a nonlinear dynamical system. Building on these insights, we establish both first-order and second-order convergence guarantees for \textbf{EXTRA} in the context of smooth and non-convex optimization:
\begin{enumerate}
    \item In Theorem \ref{the: EXTRA convergence}, we model \textbf{EXTRA} as a nonlinear dynamic system. By assuming a linear relationship between $\mathbf{W}$ and $\mathbf{V}$, we establish that the dynamical system defined by \textbf{EXTRA} exhibits asymptotic convergence to the set of first-order stationary points, which can be interpreted as the set being asymptotically attractive from a control perspective. This result extends prior work to the non-convex setting.
    \item In Theorem \ref{the: EXTRA convergence 2}, we establish that under random initialization, \textbf{EXTRA} almost surely avoids consensual strict saddle points of the global objective and eventually converge to a consensual second-order stationary point.
\end{enumerate}

The above results can be found in Section \ref{sec: Main Results}. Section \ref{sec: Assumptions and Supporting Results} provides the assumptions and supporting results. Section \ref{sec: Proofs} provides complete proofs for the theoretical results. Section \ref{sec: Numerical Examples} provides a numerical example.

\subsection{Notation}
\label{sec: Notation}
Let $\mathbf{I}_n$ denote the $n\times n$ identity matrix, $\bm{1}_n$ denote the $n$-vector with all entries equal to $1$, and $\mathbf{A}_{ij}$ denote the entry in row $i$ and column $j$ of the matrix $\mathbf{A}$. For a square symmetric matrix $\mathbf{B}$, we use $\lambda_{\mathrm{min}}(\mathbf{B})$, $\lambda_{\mathrm{max}}(\mathbf{B})$, and $\|\mathbf{B}\|$ to denote its minimum eigenvalue, maximum eigenvalue, and spectral norm, respectively. For a square matrix $\mathbf{C}$, we use $\lambda_{i}(\mathbf{C}) \ge 0$ to denote its $i$-th largest eigenvalue in magnitude. The distance from the point $\mathbf{x} \in \mathbb{R}^{n}$ to a given set $\mathcal{Y} \subseteq \mathbb{R}^{n}$ is denoted by $\textup{dist}(\mathbf{x}, \mathcal{Y})\coloneqq \inf_{\mathbf{y} \in \mathcal{Y}}\|\mathbf{x} - \mathbf{y}\|$. The Kronecker product is denoted by $\otimes$. Unless explicitly stated otherwise, all iterated parameters in this paper are positive integers.

%% file: Sections/Assumptions_and_Supporting_Results.tex
\subsection{Definitions and Assumptions}
\begin{definition}
\label{def: first-order stationary point}
    For a differentiable function $h:\mathbb{R}^{n}\rightarrow\mathbb{R}$, a point $\mathbf{x}$ is said to be first-order stationary if $\|\nabla h(\mathbf{x})\| = 0$, where $\nabla h$ denotes the gradient of $h$.
\end{definition}

\begin{assumption}[Lipschitz continuity]
\label{ass: Lipschitz}
    Each $f_{i}$ in \eqref{eq: Unconstrained optimization problem} is $L_{f_{i}}$-gradient Lipschitz, i.e., for all $\mathbf{x},\mathbf{y}\in\mathbb{R}^{n}$ and each $i \in \mathcal{V}$, $\|\nabla f_{i}(\mathbf{x}) - \nabla f_{i}(\mathbf{y})\| \le L_{f_{i}} \|\mathbf{x}-\mathbf{y}\|$.
\end{assumption}

\begin{assumption}[Coercivity and properness]
\label{ass: Coercivity}
    Each $f_{i}$ in \eqref{eq: Unconstrained optimization problem} is coercive (i.e., its sublevel set is compact) and proper (i.e., not everywhere infinite).
\end{assumption}


\begin{assumption}[Connectivity]
\label{ass: Network}
    The undirected network graph $\mathcal{G}(\mathcal{V},\mathcal{E})$ is connected.
\end{assumption}

\begin{assumption}[Mixing matrix]
\label{ass: Mixing matrix}
    The mixing matrices of network $\mathcal{G}$, $\mathbf{W} \in \mathbb{R}^{m\times m}$ and $\mathbf{V} \in \mathbb{R}^{m\times m}$ satisfy:
    \begin{itemize}
        \item If $(i,j) \notin \mathcal{E} $ and $i \neq j$, then $\mathbf{W}_{ij} = 0$.
        \item $\mathbf{W} = \mathbf{W}^{\top}$.
        \item $-\mathbf{I}_{m} \prec \mathbf{W} \preceq \mathbf{I}_{m}$ and $\textup{null}\{\mathbf{I}_{m} - \mathbf{W} \} = \textup{span}\{\bm{1}\}$.
        \item $\mathbf{V} = \theta \mathbf{I}_{m} + (1-\theta) \mathbf{W} \succ \bm{0}$, with $\theta \in (0,\frac{1}{2}]$.
    \end{itemize}    
\end{assumption}
    
\begin{remark}
    The fourth condition in Assumption \ref{ass: Mixing matrix} can be extended to linear combination of $\mathbf{I}_m$ and high-order terms of $\mathbf{W}$, $i.e.$, $\mathbf{V} = \sum_{i=0}^{n} \theta_{i} \mathbf{W}^{i}$, where $\sum_{i=1}^{n} \theta_{i} = 1$, and our results still hold. For example, \textbf{DIGing} (or Gradient Tracking) algorithm in \cite{nedic2017achieving} uses $\mathbf{V} = \frac{1}{4} \mathbf{W}^{2} + \frac{1}{2} \mathbf{W} + \frac{1}{4} \mathbf{I}$. However, the high-order terms of $\mathbf{W}$ require multiple communication steps at each iteration of the distributed algorithm.
\end{remark}

\subsection{Algorithm Review and Supporting Results}
To study distributed algorithms, we first need to reformulate the objective function in \eqref{eq: Unconstrained optimization problem} as
\begin{gather}
    \begin{gathered}
        \label{eq: Constrained optimization problem}
        F:~(\mathbb{R}^{n})^{m} \rightarrow \mathbb{R},~F(\hat{\mathbf{x}}) \triangleq \sum_{i=1}^{m} f_{i}(\hat{\mathbf{x}}_{i}),
    \end{gathered}    
\end{gather}
where $\hat{\mathbf{x}} = [\hat{\mathbf{x}}_{1}^{\top},\cdots, \hat{\mathbf{x}}_{m}^{\top}]^{\top} \in (\mathbb{R}^{n})^{m}$ with each $\hat{\mathbf{x}}_{i} \in \mathbb{R}^{n}$. Note that, $\nabla F(\hat{\mathbf{x}}) = [\nabla f_{1}(\hat{\mathbf{x}}_{1})^{\top},\cdots, \nabla f_{m}(\hat{\mathbf{x}}_{m})^{\top}]^{\top}$ and $\nabla^{2} F(\hat{\mathbf{x}}) = \bigoplus_{i=1}^{m} \nabla^{2} f_{i}(\hat{\mathbf{x}}_{i})$ where $\bigoplus$ denotes the block diagonal concatenation of matrices. In particular, the Hessian of $F$ is block diagonal. Then, the following result can be derived directly.

\begin{proposition}
    Let Assumption \ref{ass: Lipschitz} hold. $F$ defined in \eqref{eq: Constrained optimization problem} has $L_{F}$-Lipschitz continuous gradient with $L_{F} = \max_{i} \{L_{f_{i}}\}$.    
\end{proposition}

Next, we define the \textit{consensual first-order stationary} point for $F$ in \eqref{eq: Constrained optimization problem}. 
\begin{definition}
\label{def: consensual first-order stationary point}
    For function $F$ in \eqref{eq: Constrained optimization problem}, a point $\hat{\mathbf{x}} \in (\mathbb{R}^{n})^{m}$ is said to be a \textit{consensual first-order stationary} point if
    \begin{enumerate}
        \item $\hat{\mathbf{x}} = \frac{1}{m} (\bm{1}_{m}\bm{1}_{m}^{\top} \otimes \mathbf{I}_{n})\hat{\mathbf{x}}$;
        \item $(\bm{1}_{m} \otimes \mathbf{I}_{n})^{\top} \nabla F(\hat{\mathbf{x}}) = \sum_{i=1}^{m} \nabla f_{i}(\hat{\mathbf{x}}_{i}) = 0$.
    \end{enumerate}
\end{definition}

If condition 1) holds, then the point $\hat{\mathbf{x}}$ is in consensus, $i.e.$, $\hat{\mathbf{x}}_{i} = \hat{\mathbf{x}}_{j}$ for all $i$, $j \in \mathcal{V}$, where $\hat{\mathbf{x}}_{i} \in \mathbb{R}^{n}$ is the local copy of the decision vector $\mathbf{x}$ at agent $i \in \mathcal{V}$. Furthermore, if condition 2) also holds, then for any $i \in \mathcal{V}$, $\hat{\mathbf{x}}_{i}$ is a first-order stationary point of $f$ (see Proposition 2.1 in \cite{shi2015extra}).

Recall the standard fixed step-size \textbf{EXTRA} algorithm as in \eqref{eq: EXTRA}, and it can be written in an aggregate form: $ \hat{\mathbf{x}}^{1} = \hat{\mathbf{W}} \hat{\mathbf{x}}^{0} - \alpha \nabla F(\hat{\mathbf{x}}^{0})$, $\hat{\mathbf{x}}^{k+2} = \hat{\mathbf{x}}^{k+1} + \hat{\mathbf{W}} \hat{\mathbf{x}}^{k+1} - \hat{\mathbf{V}} \hat{\mathbf{x}}^{k}
        - \alpha [\nabla F(\hat{\mathbf{x}}^{k+1}) - \nabla F(\hat{\mathbf{x}}^{k})]$,
where $\hat{\mathbf{W}} = \mathbf{W} \otimes \mathbf{I}_n$ and $\hat{\mathbf{V}} = \mathbf{V} \otimes \mathbf{I}_n$. Let $\hat{\mathbf{z}}^{k} = \hat{\mathbf{x}}^{k+1} - \hat{\mathbf{W}} \hat{\mathbf{x}}^{k}$. Then, with $\hat{\mathbf{z}}^{0} = -\alpha \nabla F(\hat{\mathbf{x}}^{0})$, \textbf{EXTRA} can be formulated as a dynamical system:
\begin{gather}
\label{eq: EXTRA hybrid Jacobi-Gauss-Seidel}
    \begin{gathered}
        \hat{\mathbf{x}}^{k+1} = \hat{\mathbf{W}} \hat{\mathbf{x}}^{k}  + \hat{\mathbf{z}}^{k},\\
        \hat{\mathbf{z}}^{k+1} = (\hat{\mathbf{W}} - \hat{\mathbf{V}}) \hat{\mathbf{x}}^{k} + \hat{\mathbf{z}}^{k} - \alpha [\nabla F(\hat{\mathbf{x}}^{k+1}) - \nabla F(\hat{\mathbf{x}}^{k})].   
    \end{gathered}
\end{gather}

%% file: Sections/Main_Results.tex
We study the first-order and second-order convergence properties of \textbf{EXTRA} using the dynamical model in \eqref{eq: EXTRA hybrid Jacobi-Gauss-Seidel}.

\subsection{Properties of Mixing Matrices}
From the nonlinear system in \eqref{eq: EXTRA hybrid Jacobi-Gauss-Seidel}, we derive key mixing matrix properties used in proving Theorem~\ref{the: EXTRA convergence}.
\begin{lemma}
    \label{lem: Mixing matrix}
    Let Assumptions \ref{ass: Mixing matrix} hold. Let
    \begin{gather}
    \label{eq: P matrix}
        \mathbf{P} = 
        \begin{bmatrix}
            \mathbf{W} - \frac{1}{m} \bm{1}_{m}\bm{1}_{m}^{\top} & \mathbf{I}_{m}\\
            \mathbf{W}-\mathbf{V} & \mathbf{I}_{m} - \frac{1}{m} \bm{1}_{m}\bm{1}_{m}^{\top}
        \end{bmatrix},
    \end{gather}
    then $\lambda_{1}(\mathbf{P}) < 1$ holds.
\end{lemma}

\subsection{Convergence guarantees of \textbf{EXTRA}}
The following main results establish the first-order and second-order guarantees for \textbf{EXTRA}. Theorem \ref{the: EXTRA convergence} shows convergence to a consensual first-order stationary point of $F$. 

\begin{theorem}
\label{the: EXTRA convergence}
    Let Assumptions \ref{ass: Lipschitz}-\ref{ass: Mixing matrix} holds. For any fixed step-size
    \begin{gather}
    \label{eq: alpha range 1}
        0 < \alpha < \frac{1-(\lambda_{1}(\mathbf{P}))^{2}}{6(L_{F})^{4} + 6(L_{F})^{3} + 48 L_{F} + 1},
    \end{gather}
    where $\mathbf{P}$ is defined in \eqref{eq: P matrix}, it holds that the sequence $\{\hat{\mathbf{x}}^{k}\}$ generated by \eqref{eq: EXTRA} satisfies that
    \begin{gather*}
        \lim_{k\rightarrow\infty} \norm{(\mathbf{I}_{mn} - \frac{1}{m} (\bm{1}_{m}\bm{1}_{m}^{\top}) \otimes \mathbf{I}_{n}) \hat{\mathbf{x}}^{k}} = 0,\\
        \lim_{k\rightarrow\infty} \norm{\frac{1}{m} (\bm{1}_{m}\bm{1}_{m}^{\top}) \otimes \mathbf{I}_{n} \cdot \nabla F(\hat{\mathbf{x}}^{k})} = 0.
    \end{gather*}
\end{theorem}

\begin{remark}
    Furthermore, the final step in the proof of Theorem~\ref{the: EXTRA convergence} also shows the following convergence rates:
    \begin{gather*}
        \frac{1}{k} \sum_{\kappa=0}^{k-1}\norm{(\mathbf{I}_{mn} - \frac{1}{m} (\bm{1}_{m}\bm{1}_{m}^{\top}) \otimes \mathbf{I}_{n}) \hat{\mathbf{x}}^{\kappa}}^{2} = \mathcal{O}(\frac{1}{k}),\\
        \frac{1}{k} \sum_{\kappa=0}^{k-1}\norm{\frac{1}{m} (\bm{1}_{m}\bm{1}_{m}^{\top}) \otimes \mathbf{I}_{n} \cdot \nabla F(\hat{\mathbf{x}}^{\kappa})}^{2} = \mathcal{O}(\frac{1}{k}).
    \end{gather*}
\end{remark}

We define the \textit{consensual second-order stationary} point for $F$ in \eqref{eq: Constrained optimization problem}, followed by our main result on the second-order properties of \textbf{EXTRA}.

\begin{definition}
\label{def: consensual second-order stationary point}
    For function $F$ in \eqref{eq: Constrained optimization problem}, a point $\hat{\mathbf{x}} \in (\mathbb{R}^{n})^{m}$ is said to be a \textit{consensual second-order stationary} point if it satisfies the following:
    \begin{enumerate}
        \item $\hat{\mathbf{x}} = \frac{1}{m} (\bm{1}_{m}\bm{1}_{m}^{\top} \otimes \mathbf{I}_{n})\hat{\mathbf{x}}$;
        \item $(\bm{1}_{m} \otimes \mathbf{I}_{n})^{\top} \nabla F(\hat{\mathbf{x}}) = \sum_{i=1}^{m} \nabla f_{i}(\hat{\mathbf{x}}_{i}) = 0$;
        \item $(\bm{1}_{m} \otimes \mathbf{I}_{n})^{\top} \nabla^{2} F(\hat{\mathbf{x}}) (\bm{1}_{m} \otimes \mathbf{I}_{n}) = \sum_{i=1}^{m} \nabla^{2} f_{i}(\hat{\mathbf{x}}_{i}) \succeq 0$.
    \end{enumerate}
    Let $\mathcal{X}_{c2}$ denote the set of \textit{ consensual second-order stationary} points of $F$.
\end{definition}

Theorem \ref{the: EXTRA convergence 2} demonstrates that \textbf{EXTRA} converges almost surely to a consensual second-order stationary point of $F$ under random initialization. This result is based on a nonatomic probability measure \cite{daneshmand2020second}, which refers to a probability measure that assigns zero probability to every individual point.
\begin{theorem}
\label{the: EXTRA convergence 2}
    Let Assumptions \ref{ass: Lipschitz}-\ref{ass: Mixing matrix} hold. Suppose the initialization $\hat{\mathbf{x}}^{0}$ is drawn according to a nonatomic probability measure. For any fixed step-size
    \begin{gather*}
        0 < \alpha < \min\{\frac{1-(\lambda_{1}(\mathbf{P}))^{2}}{6(L_{F})^{4} + 6(L_{F})^{3} + 48 L_{F} + 1},~\frac{\lambda_{\min}(\mathbf{V})}{L_{F}}\},
    \end{gather*}
    it holds that the sequence $\{\hat{\mathbf{x}}^{k}\}$ generated by \eqref{eq: EXTRA} satisfies that
    \begin{gather*}
        \mathbb{P} \left[\lim_{k\rightarrow\infty} \text{dist}(\hat{\mathbf{x}}^{k},\mathcal{X}_{c2}) = 0\right] = 1,
    \end{gather*}
    where the probability is taken over $\hat{\mathbf{x}}^{0} \in (\mathbb{R}^{n})^{m}$.
\end{theorem}

%% file: Sections/Proofs.tex
\subsection{Proof of Lemma \ref{lem: Mixing matrix}}
\begin{proof}
    Let $\mathbf{P} = \mathbf{M} - \mathbf{N}$, where
    \begin{gather*}
        \begin{gathered}
            \mathbf{M} = 
            \begin{bmatrix}
                \mathbf{W} & \mathbf{I}_{m}\\
                \mathbf{W} -\mathbf{V} & \mathbf{I}_{m}
            \end{bmatrix},~
            \mathbf{N} = 
            \begin{bmatrix}
                \frac{1}{m} \bm{1}_{m}\bm{1}_{m}^{\top} & \bm{0}\\
                \bm{0} & \frac{1}{m} \bm{1}_{m}\bm{1}_{m}^{\top}
            \end{bmatrix}.
        \end{gathered}
    \end{gather*}
    To analyze the eigenvalue of $\mathbf{P} = \mathbf{M} - \mathbf{N}$, we first  analyze the eigenvalue of $\mathbf{M}$ by $\det{(\mathbf{M}-\lambda \cdot \mathbf{I}_{2m})} = 0$ with
    \begin{gather}
    \label{eq: det of M}
        \det{(\mathbf{M}-\lambda \mathbf{I}_{2m})}
        = \det{\left(
        \begin{bmatrix}
            \mathbf{W} - \lambda \mathbf{I}_{m} & \mathbf{I}_{m}\\
            \mathbf{W} -\mathbf{V} & \mathbf{I}_{m} - \lambda \mathbf{I}_{m}
        \end{bmatrix}
        \right)}.
    \end{gather}
    Applying Schur Complement, $\det{(\mathbf{M}-\lambda \mathbf{I}_{2m})} = (1-\lambda)^{m} \cdot \det{(\mathbf{W} - \lambda \mathbf{I}_{m} - \frac{1}{1-\lambda}\theta(\mathbf{W}-\mathbf{I}_{m}))} = (1-\lambda)^{m} \cdot \prod_{i=1}^{m} \omega^{i}(1-\frac{\theta}{1-\lambda}) + (\frac{\theta}{1-\lambda} - \lambda)$ by Assumption \ref{ass: Mixing matrix}, where $\omega^{i} \in (-1,1]$ is the $i$-th eigenvalue of $\mathbf{W}$ and $\theta \in (0,\frac{1}{2}]$. \textbf{Case 1}: For $i \in \mathcal{V}$ such that $\omega^{i} = 1$, $1$ is  the only one solution to $\det{(\mathbf{M}-\lambda \mathbf{I}_{2m})} = 0$. \textbf{Case 2}: For $i \in \mathcal{V}$ such that $\omega^{i} \ne 1$ and $(1 - \omega^{i})^2-4(1 - \omega^{i})\theta \ge 0$, all solutions satisfy
    \begin{gather}
    \label{eq: Determinant}
        \prod_{i=1}^{m} \lambda^{2} - (1+\omega^{i}) \lambda + \omega^{i}-\theta\omega^{i} + \theta = 0,
    \end{gather}
    and are all real. Since $1 - \omega^{i} > 0$ for all $i \in \mathcal{V}$ and $\theta \in (0,\frac{1}{2}]$, then $\abs{\lambda} \le \frac{1 + \omega^{i} + \sqrt{(1 - \omega^{i})^2-4(1 - \omega^{i})\theta}}{2} < \frac{1 + \omega^{i} + \sqrt{(1 - \omega^{i})^2}} {2} \le 1$. \textbf{Case 3}: For $i \in \mathcal{V}$ such that $\omega^{i} \ne 1$ and $(1 - \omega^{i})^2-4(1 - \omega^{i})\theta < 0$, all solutions satisfy \eqref{eq: Determinant} and are all complex. Then, $\abs{\lambda} \le \omega^{i} + \theta (1 - \omega^{i}) \le \frac{1}{2} + \frac{1}{2}\omega^{i} < 1$. Combining the three cases above, we conclude that all eigenvalues of \( \mathbf{M} \) have magnitude strictly less than 1, except for a simple eigenvalue at 1. We now examine the eigenvector associated with this eigenvalue. Suppose \( \lambda = 1 \), so there exists a nonzero vector \( \mathbf{v} = [\mathbf{v}_1^{\top}, \mathbf{v}_2^{\top}]^{\top} \in \mathbb{R}^{2m} \) such that \( \mathbf{M} \mathbf{v} = \mathbf{v} \). This implies \( (\mathbf{W} - \mathbf{V}) \mathbf{v}_1 = \bm{0} \), and by Assumption~\ref{ass: Mixing matrix}, \( (\mathbf{I}_m - \mathbf{W}) \mathbf{v}_1 = \bm{0} \). Using \( \mathbf{W} \mathbf{v}_1 + \mathbf{v}_2 = \mathbf{v}_1 \), we obtain the eigenvector \( \mathbf{v} = [\tfrac{1}{\sqrt{m}} \bm{1}^{\top}, \bm{0}^{\top}]^{\top} \). This is also an eigenvector of \( \mathbf{N} \) with eigenvalue 1, while all other eigenvectors of \( \mathbf{N} \) correspond to eigenvalue 0. Thus, subtracting \( \mathbf{N} \) from \( \mathbf{M} \) eliminates the eigenvalue at 1, yielding \( \lambda_1(\mathbf{P}) < 1 \).
\end{proof}

\subsection{Proof of Theorem \ref{the: EXTRA convergence}}
\begin{proof}
    We first make a change of coordinates to extract the consensus error dynamics, specifically the evolution of the distance between the agents' iterates and their average. Let $\bar{\mathbf{x}}^{k} = \hat{\mathbf{x}}^{k} - \tilde{\mathbf{x}}^{k}$, and 
    $\bar{\mathbf{z}}^{k} = \hat{\mathbf{z}}^{k} - \tilde{\mathbf{z}}^{k}$, where $\tilde{\mathbf{x}}^{k} = \mathbf{I}_{mn}^{\prime} \hat{\mathbf{x}}^{k}$, $\tilde{\mathbf{z}}^{k} = \mathbf{I}_{mn}^{\prime} \hat{\mathbf{z}}^{k}$ with $\mathbf{I}_{mn}^{\prime} = (\mathbf{I}_{mn} - \frac{1}{m} (\bm{1}_{m}\bm{1}_{m}^{\top}) \otimes \mathbf{I}_{n})$. Then multiplying $\mathbf{I}_{mn}^{\prime}$ on the both sides of \eqref{eq: EXTRA hybrid Jacobi-Gauss-Seidel} gives
    \begin{gather}
    \label{eq: Tilde dynamic}
        \begin{bmatrix}
            \tilde{\mathbf{x}}^{k+1}\\
            \tilde{\mathbf{z}}^{k+1}
        \end{bmatrix}
        = \hat{\mathbf{P}} 
        \begin{bmatrix}
            \tilde{\mathbf{x}}^{k}\\
            \tilde{\mathbf{z}}^{k}
        \end{bmatrix}
        - \alpha \mathbf{I}_{mn}^{\prime}
        \begin{bmatrix}
            \bm{0}\\
            \nabla F(\hat{\mathbf{x}}^{k+1}) - \nabla F(\hat{\mathbf{x}}^{k})
        \end{bmatrix},
    \end{gather}
    where $\hat{\mathbf{P}} \coloneqq \mathbf{P} \otimes \mathbf{I}_n$, as $\mathbf{P}$ defined in \eqref{eq: P matrix}. Let $G:~(\mathbb{R}^{n})^{m} \rightarrow \mathbb{R}$ and $G(\hat{\mathbf{x}}) \triangleq F(\frac{1}{m}(\bm{1}_{m}\bm{1}_{m}^{\top}) \otimes \mathbf{I}_{n} \cdot \hat{\mathbf{x}})$. Then, by Assumption \ref{ass: Lipschitz}, $G(\hat{\mathbf{x}})$ also has $L_{F}$-Lipschitz continuous gradient. Multiplying $\frac{1}{m} (\bm{1}_{m}\bm{1}_{m}^{\top}) \otimes \mathbf{I}_{n}$ on the both sides of \eqref{eq: EXTRA hybrid Jacobi-Gauss-Seidel} gives
    \begin{align}
    \label{eq: Bar dynamic}
        \begin{aligned}
            \bar{\mathbf{x}}^{k+1} &= \bar{\mathbf{x}}^{k} - \frac{\alpha}{m} (\bm{1}_{m}\bm{1}_{m}^{\top}) \otimes \mathbf{I}_{n} \cdot \nabla F(\hat{\mathbf{x}}^{k})\\
            &= \bar{\mathbf{x}}^{k} - \frac{\alpha}{m} (\bm{1}_{m}\bm{1}_{m}^{\top}) \otimes \mathbf{I}_{n} \cdot \nabla F(\bar{\mathbf{x}}^{k})\\
            &\quad- \frac{\alpha}{m} (\bm{1}_{m}\bm{1}_{m}^{\top}) \otimes \mathbf{I}_{n} \cdot (\nabla F(\hat{\mathbf{x}}^{k}) - \nabla F(\bar{\mathbf{x}}^{k}))\\
            &= \bar{\mathbf{x}}^{k} - \alpha \nabla G(\bar{\mathbf{x}}^{k}) - \alpha \Delta_{g}^{k},
        \end{aligned}
    \end{align}
    where $\Delta_{g}^{k} = \frac{1}{m} (\bm{1}_{m}\bm{1}_{m}^{\top}) \otimes \mathbf{I}_{n} \cdot (\nabla F(\hat{\mathbf{x}}^{k}) - \nabla F(\bar{\mathbf{x}}^{k}))$. 
    Since $\norm{\hat{\mathbf{W}} - \frac{1}{m} (\bm{1}_{m}\bm{1}_{m}^{\top}) \otimes \mathbf{I}_{n} - \mathbf{I}_{mn}} \le 2$, it follows
    \begin{align}
    \label{eq: xk+1 and xk}
    \begin{aligned}
         &\norm{\hat{\mathbf{x}}^{k+1} - \hat{\mathbf{x}}^{k}} \le \norm{\bar{\mathbf{x}}^{k+1} - \bar{\mathbf{x}}^{k}} + \norm{\tilde{\mathbf{x}}^{k+1} - \tilde{\mathbf{x}}^{k}}\\
        &\quad\quad\overset{\eqref{eq: Tilde dynamic},~\eqref{eq: Bar dynamic}}{\le} \norm{\frac{1}{m}( (\bm{1}_{m}\bm{1}_{m}^{\top}) \otimes \mathbf{I}_{n}) \nabla F(\hat{\mathbf{x}}^{k})}\\
        &\quad\quad\quad+ \norm{(\hat{\mathbf{W}} - \frac{1}{m} (\bm{1}_{m}\bm{1}_{m}^{\top}) \otimes \mathbf{I}_{n} - \mathbf{I}_{mn}) \tilde{\mathbf{x}}^{k} + \tilde{\mathbf{z}}^{k}}\\
        &\quad\quad\le \norm{\frac{1}{m}( (\bm{1}_{m}\bm{1}_{m}^{\top}) \otimes \mathbf{I}_{n}) \nabla F(\hat{\mathbf{x}}^{k})} + 3 \norm{\tilde{\mathbf{\xi}}^{k}},
    \end{aligned}
    \end{align}
    where $\tilde{\mathbf{\xi}}^{k} = [(\tilde{\mathbf{x}}^{k})^{\top},(\tilde{\mathbf{z}}^{k})^{\top}]^{\top}$. Recall \eqref{eq: Tilde dynamic}. By Assumption \ref{ass: Lipschitz},
    \begin{align*}
         \norm{\tilde{\mathbf{\xi}}^{k+1}} &\overset{\eqref{eq: xk+1 and xk}}{\le} (\lambda_{1}(\mathbf{P}) + 3\alpha L_{F})
        \norm{\tilde{\mathbf{\xi}}^{k}}\\
        &\quad+ \alpha L_{F} \norm{\frac{1}{m}( (\bm{1}_{m}\bm{1}_{m}^{\top}) \otimes \mathbf{I}_{n}) (\nabla F(\hat{\mathbf{x}}^{k}))}\\
        &\le (\lambda_{1}(\mathbf{P}) + 3\alpha L_{F}) \norm{\tilde{\mathbf{\xi}}^{k}} + \alpha L_{F} \norm{\nabla G(\bar{\mathbf{x}}^{k})}\\
        &\quad+ \alpha L_{F} \norm{\frac{1}{m} ((\bm{1}_{m}\bm{1}_{m}^{\top}) \otimes \mathbf{I}_{n}) (\nabla F(\hat{\mathbf{x}}^{k}) - \nabla F(\bar{\mathbf{x}}^{k}))}\\
        &\le (\lambda_{1}(\mathbf{P}) + 4\alpha L_{F})
        \norm{\tilde{\mathbf{\xi}}^{k}} + \alpha L_{F} \norm{\nabla G(\bar{\mathbf{x}}^{k})}.
    \end{align*}
    By Young’s inequality $ab\le \frac{a^2}{2\epsilon} + \frac{\epsilon b^2}{2}$, for any fixed $\epsilon > 0$,
    \begin{multline*}
        \norm{\tilde{\mathbf{\xi}}^{k+1}}^{2} \le (1+\epsilon)(\lambda_{1}(\mathbf{P}) + 4\alpha L_{F})^{2}
        \norm{\tilde{\mathbf{\xi}}^{k}}^{2}\\
        + (1+\frac{1}{\epsilon}) (\alpha L_{F})^{2} \norm{\nabla G(\bar{\mathbf{x}}^{k})}^{2}.
    \end{multline*}
    Summing over $k$ gives that for any $k > 0$,
    \begin{multline}
    \label{eq: Sum of errors}
        \sum_{\kappa=0}^{k} \norm{\begin{bmatrix}
            \tilde{\mathbf{x}}^{\kappa}\\
            \tilde{\mathbf{z}}^{\kappa}
        \end{bmatrix}}^{2} \le (1+\epsilon)(\lambda_{1}(\mathbf{P}) + 4\alpha L_{F})^{2} \sum_{\kappa=0}^{k-1}
        \norm{\tilde{\mathbf{\xi}}^{\kappa}}^{2}\\
        + (1+\frac{1}{\epsilon}) (\alpha L_{F})^{2} \sum_{\kappa=0}^{k-1}  \norm{\nabla G(\bar{\mathbf{x}}^{\kappa})}^{2} + \norm{\tilde{\mathbf{\xi}}^{0}}^{2}.
    \end{multline}
    Again by Taylor’s theorem and Young's inequality $ab\le \frac{a^2}{2\rho} + \frac{\rho b^2}{2}$ with $\rho = 2L_{F}$, it holds that for any fixed $0 < \alpha \le \frac{1}{L_{F}}$,
    \begin{align*}
        &G(\bar{\mathbf{x}}^{k+1}) - G(\bar{\mathbf{x}}^{k})\\
        &\quad\le \nabla G(\bar{\mathbf{x}}^{k})^{\top}(\bar{\mathbf{x}}^{k+1} - \bar{\mathbf{x}}^{k}) + \frac{L_{F}}{2} \norm{\bar{\mathbf{x}}^{k+1} - \bar{\mathbf{x}}^{k}}^{2}\\
        &\quad\overset{\eqref{eq: Bar dynamic}}{\le} -\alpha \norm{\nabla G(\bar{\mathbf{x}}^{k})}^{2} - \alpha \nabla G(\bar{\mathbf{x}}^{k})^{\top} \Delta_{g}^{k}\\
        &\quad\quad + \frac{L_{F} \alpha^{2}}{2} \norm{\nabla G(\bar{\mathbf{x}}^{k}) + \Delta_{g}^{k}}^{2}\\
        &\quad\le (-\alpha+\frac{L_{F}\alpha^{2}} {2}) \norm{\nabla G(\bar{\mathbf{x}}^{k})}^{2} + (\alpha - L_{F} \alpha^{2})L_{F}\\
        &\quad\quad \norm{\nabla G(\bar{\mathbf{x}}^{k})} \cdot\norm{\tilde{\mathbf{x}}^{k}} + \frac{(L_{F})^{3}\alpha^{2}}{2} \norm{\tilde{\mathbf{x}}^{k}}^{2}\\
        &\quad\le -\frac{\alpha}{2} \norm{\nabla G(\bar{\mathbf{x}}^{k})}^{2} + \alpha L_{F} \norm{\nabla G(\bar{\mathbf{x}}^{k})} \cdot \norm{\tilde{\mathbf{x}}^{k}}\\
        &\quad\quad+ \frac{(L_{F})^{3}\alpha^{2}}{2} \norm{\tilde{\mathbf{x}}^{k}}^{2}\\
        &\quad\le -\frac{\alpha}{4} \norm{\nabla G(\bar{\mathbf{x}}^{k})}^{2} + (\frac{(L_{F})^{3}\alpha^{2}}{2} + (L_{F})^{2}\alpha) \norm{\tilde{\mathbf{x}}^{k}}^{2}.
    \end{align*}
    Summing over $k$ gives that for any $k > 0$,
    \begin{multline}
    \label{eq: Sum of gradients}
        \sum_{\kappa=0}^{k-1} \norm{\nabla G(\bar{\mathbf{x}}^{\kappa})}^{2} \le \frac{4}{\alpha} (G(\bar{\mathbf{x}}^{0}) - G(\bar{\mathbf{x}}^{\star}))\\
        + 4 \left(\frac{(L_{F})^{3}\alpha}{2} + (L_{F})^{2}\right) \sum_{\kappa=0}^{k-1}  \norm{\tilde{\mathbf{\xi}}^{\kappa}}^{2}.
    \end{multline}
    Let $S^{k} = \sum_{\kappa=0}^{k} \norm{\tilde{\mathbf{\xi}}^{\kappa}}^{2}$. Combining \eqref{eq: Sum of errors} and \eqref{eq: Sum of gradients} gives
    \begin{gather}
    \label{eq: Sum of errors system}
        S^{k+1} \le \lambda(\epsilon,\alpha) \cdot S^{k} + u(\epsilon,\alpha),
    \end{gather}
    where 
    \begin{multline*}
        \lambda(\epsilon,\alpha) = (1+\epsilon)(\lambda_{1}(\mathbf{P}) + 4\alpha L_{F})^{2}\\
        +  4(1+\frac{1}{\epsilon}) (\alpha L_{F})^{2} \left(\frac{(L_{F})^{3}\alpha}{2} + (L_{F})^{2}\right),
    \end{multline*}
    and 
    \begin{gather*}
        u(\epsilon,\alpha) = 4(1+\frac{1}{\epsilon})( L_{F})^{2} \alpha (G(\bar{\mathbf{x}}^{0}) - G(\bar{\mathbf{x}}^{\star})) + \norm{\tilde{\mathbf{\xi}}^{0}}^{2}.
    \end{gather*}
    Then, choosing $\epsilon = \alpha$ yields that for any fixed $\alpha > 0$ satisfying \eqref{eq: alpha range 1},
    \begin{gather*}
        0 < \alpha < \frac{1-(\lambda_{1}(\mathbf{P}))^{2}}{6(L_{F})^{4} + 6(L_{F})^{3} + 48 L_{F} + 1} < \frac{1}{L_{F}},
    \end{gather*}
    and
    \begin{gather*}
        \lambda(\epsilon,\alpha) < 1,~u(\epsilon,\alpha) \le 8(L_{F})^{2} (G(\bar{\mathbf{x}}^{0}) - G(\bar{\mathbf{x}}^{\star})) + \norm{\tilde{\mathbf{\xi}}^{0}}^{2}.
    \end{gather*}
    By Assumption \ref{ass: Coercivity}, $u(\epsilon,\alpha)$ is lower bounded and
    \begin{gather*}
        S^{k} \overset{\eqref{eq: Sum of errors system}}{\le} \lambda(\epsilon,\alpha)^{k} \cdot S^{0} + \frac{u(\epsilon,\alpha)}{1-\lambda(\epsilon,\alpha)}
    \end{gather*}
    implying $S^{k} = \sum_{\kappa=0}^{k} \norm{\tilde{\mathbf{\xi}}^{\kappa}}^{2}$ 
    and $\sum_{\kappa=0}^{k-1} \norm{\nabla G(\bar{\mathbf{x}}^{\kappa})}^{2}$ are uniformly bounded for any $k > 0$. Then, 
    \begin{multline*}
        \sum_{\kappa=0}^{k-1}\norm{\frac{1}{m} (\bm{1}_{m}\bm{1}_{m}^{\top}) \otimes \mathbf{I}_{n} \cdot \nabla F(\hat{\mathbf{x}}^{\kappa})}^{2}\\
        \le \sum_{\kappa=0}^{k-1} \norm{\nabla G(\bar{\mathbf{x}}^{\kappa})}^{2} + L_{F} \cdot \sum_{\kappa=0}^{k} \norm{\tilde{\mathbf{x}}^{\kappa}}^{2},
    \end{multline*}
    which is also uniformly bounded for any $k > 0$. Therefore, sequences $\left\{\frac{1}{k} \sum_{\kappa=0}^{k-1}\norm{\frac{1}{m} (\bm{1}_{m}\bm{1}_{m}^{\top}) \otimes \mathbf{I}_{n} \cdot \nabla F(\hat{\mathbf{x}}^{\kappa})}^{2}\right\}$ and $\left\{\frac{1}{k} \sum_{\kappa=0}^{k-1} \norm{\tilde{\mathbf{x}}^{\kappa}}^{2}\right\}$ both converge to $0$ at a rate of $\mathcal{O}(\frac{1}{k})$ as claimed, which implies the set of consensual first-order stationary points is asymptotically attractive.
\end{proof}

\subsection{Proof of Theorem \ref{the: EXTRA convergence 2}}
After establishing convergence to consensual first-order stationary points, we define the consensual strict saddle point for $F$ in \eqref{eq: Constrained optimization problem}, and show that \textbf{EXTRA} avoids converging to these points.
\begin{definition}
\label{def: consensual strict saddle point}
    For function $F$ in \eqref{eq: Constrained optimization problem}, a point $\hat{\mathbf{x}} \in (\mathbb{R}^{n})^{m}$ is said to be a consensual strict saddle point if it satisfies the following:
    \begin{enumerate}
        \item $\hat{\mathbf{x}} = \frac{1}{m} (\bm{1}_{m}\bm{1}_{m}^{\top} \otimes \mathbf{I}_{n})\hat{\mathbf{x}}$;
        \item $(\bm{1}_{m} \otimes \mathbf{I}_{n})^{\top} \nabla F(\hat{\mathbf{x}}) = \sum_{i=1}^{m} \nabla f_{i}(\hat{\mathbf{x}}_{i}) = 0$;
        \item $\lambda_{\min}((\bm{1}_{m} \otimes \mathbf{I}_{n})^{\top} \nabla^{2} F(\hat{\mathbf{x}}) (\bm{1}_{m} \otimes \mathbf{I}_{n})) < 0$,
    \end{enumerate}
    where $(\bm{1}_{m} \otimes \mathbf{I}_{n})^{\top} \nabla^{2} F(\hat{\mathbf{x}}) (\bm{1}_{m} \otimes \mathbf{I}_{n}) = \sum_{i=1}^{m} \nabla^{2} f_{i}(\hat{\mathbf{x}}_{i})$. Let $\mathcal{X}_{cs}$ denote the set of consensual strict saddle points of $F$. 
\end{definition}

To analyze the second order property, the key idea is to identify a self-map $T$ that represents \textbf{EXTRA}, such that the stable set of the strict saddle points of $F$ has zero measure within the domain of the mapping. Particularly, the proof has two main steps. First, we demonstrate that the consensual strict saddle point is an unstable fixed point of \textbf{EXTRA} (see Lemma \ref{lem: Fixed point}). Next, we establish the conditions under which the mapping $T$ is a $C^{1}$ diffeomorphis (see Lemma \ref{lem: Det}). By combining these results, we arrive at Proposition \ref{pro: EXTRA non-convergence}. To identify the corresponding mapping in \textbf{EXTRA} updates, recall \eqref{eq: EXTRA hybrid Jacobi-Gauss-Seidel}. 
By defining $\hat{\mathbf{z}}^{k} = \hat{\mathbf{y}}^{k} - \alpha \nabla F(\hat{\mathbf{x}}^{k})$, \textbf{EXTRA} in Jacobi form can be formulated as
\begin{gather}
\label{eq: EXTRA Jacobi}
    \begin{gathered}
        \hat{\mathbf{x}}^{k+1} = \hat{\mathbf{W}} \hat{\mathbf{x}}^{k} + \hat{\mathbf{y}}^{k} - \alpha \nabla F(\hat{\mathbf{x}}^{k}),\\
        \hat{\mathbf{y}}^{k+1} = (\hat{\mathbf{W}} - \hat{\mathbf{V}}) \hat{\mathbf{x}}^{k} + \hat{\mathbf{y}}^{k}        
    \end{gathered}
\end{gather}
with $\hat{\mathbf{y}}^{0} = \bm{0}$, which gives the corresponding mapping $T(\hat{\mathbf{x}},\hat{\mathbf{y}}): (\mathbb{R}^{n})^{m} \times (\mathbb{R}^{n})^{m} \rightarrow (\mathbb{R}^{n})^{m} \times (\mathbb{R}^{n})^{m}$ in \textbf{EXTRA} updates, 
\begin{gather*}
    \begin{bmatrix}
        \hat{\mathbf{x}}^{k+1}\\
        \hat{\mathbf{y}}^{k+1}
    \end{bmatrix}
    = T(\hat{\mathbf{x}}^{k},\hat{\mathbf{y}}^{k}),
\end{gather*}
where
\begin{gather}
\label{eq: Mapping}
    T(\hat{\mathbf{x}},\hat{\mathbf{y}}) = 
    \begin{bmatrix}
        \hat{\mathbf{W}} \hat{\mathbf{x}} + \hat{\mathbf{y}} - \alpha \nabla F(\hat{\mathbf{x}})\\
        (\hat{\mathbf{W}} - \hat{\mathbf{V}}) \hat{\mathbf{x}} + \hat{\mathbf{y}}
    \end{bmatrix}.
\end{gather}

We next define the unstable fixed points of the mapping.
\begin{definition}
    Given a mapping $h: \mathcal{X} \rightarrow \mathcal{X}$, let $\mathcal{A}_{h}$ be the set of the unstable fixed points of $h$, $i.e.$,
    \begin{gather*}
        \mathcal{A}_{h} \coloneqq \{\mathbf{x}:~h(\mathbf{x}) = \mathbf{x},~\lambda_{1}(\mathrm{D}h(\mathbf{x})) > 1\},
    \end{gather*}
    where $\mathrm{D}h(\mathbf{x})$ is the differential of $h$ at $\mathbf{x}$.
\end{definition}

Next, Proposition \ref{pro: EXTRA non-convergence} demonstrates that \textbf{EXTRA} avoids convergence to consensual strict saddle points of $F$.
\begin{lemma}
\label{lem: Fixed point}
    Let Assumptions \ref{ass: Lipschitz}, \ref{ass: Network}, \ref{ass: Mixing matrix} hold. Suppose $\hat{\mathbf{x}}$ is a consensual strict saddle point of $F$. For any fixed $\alpha > 0$, $[\hat{\mathbf{x}}^{\top},~\alpha \nabla F(\hat{\mathbf{x}})^{\top}]^{\top}$ is an unstable fixed point of \textbf{EXTRA}, $i.e.$, $\mathcal{X}_{cs} \times \mathcal{G}(\mathcal{X}_{cs}) \subseteq \mathcal{A}_{T}$, where $\mathcal{G}(\mathcal{X}_{cs}) = \{\nabla F(\hat{\mathbf{x}}):~\hat{\mathbf{x}} \in \mathcal{X}_{cs}\}$.
\end{lemma}

\begin{proof}
    By Assumption \ref{ass: Lipschitz} and \eqref{eq: Mapping},
    \begin{gather}
    \label{eq: Det of mapping}
        \mathrm{D}T(\hat{\mathbf{x}},\hat{\mathbf{y}}) = 
        \begin{bmatrix}
            \hat{\mathbf{W}} - \alpha \nabla^{2} F(\hat{\mathbf{x}}) &  \mathbf{I}_{mn}\\
            \hat{\mathbf{W}} - \hat{\mathbf{V}} & \mathbf{I}_{mn}
        \end{bmatrix}.
    \end{gather}
    Suppose $\hat{\mathbf{x}}$ is a consensual strict saddle point of $F$. Since $\nabla^{2} F(\hat{\mathbf{x}})$ is a block diagonal matrix, by the definition of consensual strict saddle point (see Definition \ref{def: consensual strict saddle point}), it holds that $\lambda_{\min}(\nabla^{2} F(\hat{\mathbf{x}})) \le \lambda_{\min}((\bm{1}_{m} \otimes \mathbf{I}_{n})^{\top} \nabla^{2} F(\hat{\mathbf{x}}) (\bm{1}_{m} \otimes \mathbf{I}_{n})) < 0$. To analyze the eigenvalues of $\mathrm{D}T(\hat{\mathbf{x}},\hat{\mathbf{y}})$, we focus on $\det(\mathrm{D}T(\hat{\mathbf{x}},\hat{\mathbf{y}}) - \lambda \mathbf{I}_{2mn}) = 0$.
    By Schur Completement,
    \begin{multline*}
        \det(\mathrm{D}T(\hat{\mathbf{x}},\hat{\mathbf{y}}) - \lambda \mathbf{I}_{2mn}) = (1-\lambda)^{mn} \cdot\\
        \det(\hat{\mathbf{W}} - \alpha \nabla^{2} F(\hat{\mathbf{x}}) - \lambda \mathbf{I}_{mn} - \frac{1}{1-\lambda} (\hat{\mathbf{W}} - \hat{\mathbf{V}})) = 0
    \end{multline*}
    when $\lambda \ne 1$. Therefore, it is needed to solve $(1-\lambda)^{mn} = 0$ and $\det((1-\lambda)(\hat{\mathbf{W}} - \alpha \nabla^{2} F(\hat{\mathbf{x}}) - \lambda \mathbf{I}_{mn}) +  \hat{\mathbf{V}} - \hat{\mathbf{W}}) = 0$. Define $\psi(\lambda) = \lambda^{2} \mathbf{I}_{mn} - (\hat{\mathbf{W}} + \mathbf{I}_{mn} - \alpha \nabla^{2} F(\hat{\mathbf{x}})) \lambda + \hat{\mathbf{V}} - \alpha \nabla^{2} F(\hat{\mathbf{x}})$. Trivially, for sufficient large $\bar{\lambda} > 1$, $\det(\psi(\bar{\lambda})) > 0$.
    Next, for any $\mathbf{v} \in \mathbb{R}^{n}$, by choosing $\hat{\mathbf{v}} = \bm{1}_{m} \otimes \mathbf{v}$, by involving the Rayleigh quotient, we have $R(\hat{\mathbf{v}}) = \frac{\hat{\mathbf{v}}^{\top} \psi(\lambda) \hat{\mathbf{v}}}{\|\hat{\mathbf{v}}\|^{2}}$
    implying
    \begin{multline*}
       \lambda_{\min} (\psi(\lambda)) \le R(\hat{\mathbf{v}}) \\
       = \lambda^{2} - (2-\alpha\lambda_{\min}(\nabla^{2} F(\hat{\mathbf{x}}))) \lambda + 1-\alpha\lambda_{\min}(\nabla^{2} F(\hat{\mathbf{x}})).
    \end{multline*}
    Given $\alpha > 0$, since $\lambda_{\min}(\nabla^{2} F(\hat{\mathbf{x}})) < 0$, for $\lambda = 1 - \alpha\lambda_{\min}(\nabla^{2} F(\hat{\mathbf{x}})) > 1$, $\lambda_{\min} (\psi(\lambda)) \le -\frac{1}{4} \alpha^{2} (\lambda_{\min} (\psi(\lambda)))^{2} < 0$. By the continuity of $\psi(\lambda)$, there exists $\lambda > 1$ such that $\det(\psi(\lambda)) = 0$, which means $\det(\mathrm{D}T(\hat{\mathbf{x}},\hat{\mathbf{y}}) - \lambda \mathbf{I}_{2mn}) = 0$.
\end{proof}

\begin{lemma}
\label{lem: Det}
    Let Assumptions \ref{ass: Lipschitz}, \ref{ass: Network}, \ref{ass: Mixing matrix} hold. Suppose $\hat{\mathbf{x}}$ is a consensual strict saddle point of $F$. For any fixed $0 < \alpha <\frac{\lambda_{\min}(\mathbf{V})}{L_{F}}$, $\det(\mathrm{D}T(\hat{\mathbf{x}},\alpha \nabla F(\hat{\mathbf{x}}))) \ne 0$.
\end{lemma}

\begin{proof}
By Schur Complement,
    \begin{align*}
        \det(\mathrm{D}T(\hat{\mathbf{x}},\alpha \nabla F(\hat{\mathbf{x}}))) &\overset{\eqref{eq: Det of mapping}}{=}  \det(\mathbf{I}_{mn}) \cdot \det(\hat{\mathbf{W}} - \alpha \nabla^{2} F(\hat{\mathbf{x}})\\
        &\quad+ \hat{\mathbf{V}} - \hat{\mathbf{W}})\\
        &= \det(\hat{\mathbf{V}} - \alpha \nabla^{2} F(\hat{\mathbf{x}})).
    \end{align*}
    Since $\mathbf{V} \succ \bm{0}$, choosing $\alpha < \frac{\lambda_{\min}(\mathbf{V})}{L_{F}}$ gives $\hat{\mathbf{V}} - \alpha \nabla^{2} F(\hat{\mathbf{x}}) \succ 0$, which implies $\det(\mathrm{D}T(\hat{\mathbf{x}},\alpha \nabla F(\hat{\mathbf{x}}))) \ne 0$ as claimed.
\end{proof}

\begin{proposition}
\label{pro: EXTRA non-convergence}
    Let Assumptions \ref{ass: Lipschitz}-\ref{ass: Mixing matrix} holds. Suppose the initialization $\hat{\mathbf{x}}^{0}$ is drawn according to a nonatomic probability measure.  For any fixed $0 < \alpha < \frac{\lambda_{\min}(\mathbf{V})}{L_{F}}$, it holds that the sequence $\{\hat{\mathbf{x}}^{k}\}$ generated by \textbf{EXTRA} satisfies that
    \begin{gather*}
        \mathbb{P}\left[\lim_{k \rightarrow \infty} \hat{\mathbf{x}}^{k} \in \mathcal{X}_{cs}\right] = 0,
    \end{gather*}
    where the probability is taken over $\hat{\mathbf{x}}^{0} \in (\mathbb{R}^{n})^{m}$.
\end{proposition}

\begin{proof}
    Lemma \ref{lem: Fixed point} shows that $\mathcal{X}_{cs} \times \mathcal{G}(\mathcal{X}_{cs}) \subseteq \mathcal{A}_{T}$, and Lemma \ref{lem: Det} shows that with $\alpha < \frac{\lambda_{\min}(\mathbf{V})}{L_{F}}$, $\det(\mathrm{D}T(\hat{\mathbf{x}},\alpha \nabla F(\hat{\mathbf{x}}))) \ne 0$. Applying Corollary 1 in \cite{lee2019first}  gives  that the stable set of the strict saddle points has zero measure, i.e., $\mu(\{(\hat{\mathbf{x}}^{0},\bm{0}):\lim_{k\rightarrow\infty} T^{k}(\hat{\mathbf{x}}^{0},\bm{0}) \in \mathcal{X}_{cs} \times \mathcal{G}(\mathcal{X}_{cs})\}) = 0$,
    where $\mu$ denotes the measure and $T^{k}$ is the $k$-fold composition of $T$. With any initialization $\hat{\mathbf{x}}^{0}$ drawn according to a nonatomic probability measure, we have
    \begin{gather*}
        \mathbb{P}\left[\lim_{k \rightarrow \infty} \hat{\mathbf{x}}^{k} \in \mathcal{X}_{cs}\right] = 0
    \end{gather*}
    as claimed.
\end{proof}

\begin{proofs}
    \textbf{Proof of Theorem \ref{the: EXTRA convergence 2}}: Let $\mathcal{X}_{c1}$ denote the set of consensual first-order stationary points of $F$ defined in Definition \ref{def: consensual first-order stationary point}. By Theorem \ref{the: EXTRA convergence}, for any fixed $0 < \alpha < \frac{1-(\lambda_{1}(\mathbf{P}))^{2}}{6(L_{F})^{4} + 6(L_{F})^{3} + 48 L_{F} + 1}$,
    it holds that $\lim_{k\rightarrow\infty} \text{dist}(\hat{\mathbf{x}}^{k},\mathcal{X}_{c1}) = 0$. By Proposition \ref{pro: EXTRA non-convergence}, for any fixed $0 < \alpha < \frac{\lambda_{\min}(\mathbf{V})}{L_{F}}$, it holds that $\mathbb{P}\left[\lim_{k \rightarrow \infty} \hat{\mathbf{x}}^{k} \in \mathcal{X}_{cs}\right] = 0$. By the definition of  $\mathcal{X}_{cs}$ in Definition \ref{def: consensual strict saddle point}, combining the above results gives that for any fixed $0 < \alpha < \min\{\frac{1-(\lambda_{1}(\mathbf{P}))^{2}}{6(L_{F})^{4} + 6(L_{F})^{3} + 48 L_{F} + 1},~\frac{\lambda_{\min}(\mathbf{V})}{L_{F}}\}$,
    it holds that
    \begin{gather*}
        \mathbb{P}\left[\lim_{k\rightarrow\infty} \text{dist}(\hat{\mathbf{x}}^{k},\mathcal{X}_{c2}) = 0\right] = 1
    \end{gather*} 
    as claimed.
\end{proofs}

%% file: Sections/Numerical_Examples.tex
\begin{figure}
    \centering
    \begin{subfigure}[h]{0.25\textwidth}
        \centering
        \includegraphics[width=\columnwidth]{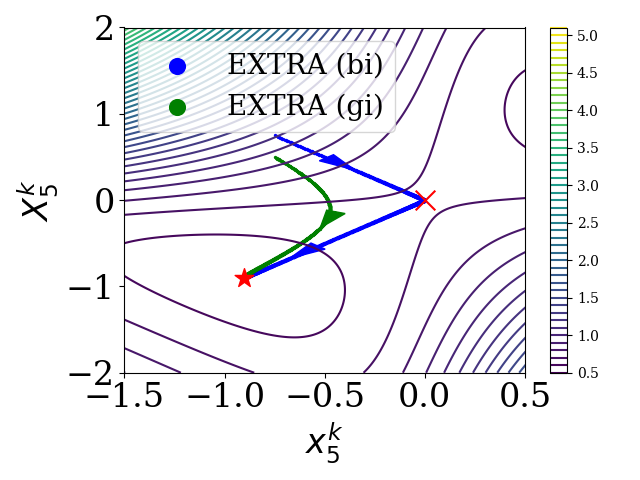}
        \caption{Trajectory of $[\mathbf{x}_{5}^{k},\mathbf{X}_{5}^{k}]$ at agent $5$, with $\times$ and $\star$ denoting the saddle point and the local minimizer.}
        \label{fig: Contour plot}
    \end{subfigure}
    \begin{subfigure}[h]{0.225\textwidth}
        \centering
        \includegraphics[width=\columnwidth]{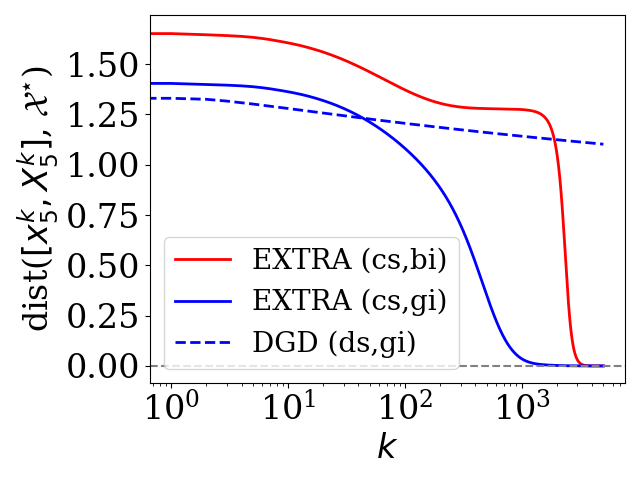}
        \caption{Evolution of distance between $[\mathbf{x}_{5}^{k},\mathbf{X}_{5}^{k}]$ and set $\mathcal{X}^{\star}$ at agent $5$.}
        \label{fig: Iteration plot}
    \end{subfigure}
    \centering
    \caption{\textbf{EXTRA} on a non-convex binary classification task. $ [\mathbf{x}_{5}^{k}, \mathbf{X}_{5}^{k}] $ denotes the $ k $-th iterate at agent $5$, and $\mathcal{X}^{\star}$ denotes the set of local minimizers (marked by a red $ \star $). ``bi'' and ``gi'' refer to bad and good initializations, respectively; ``cs'' and ``ds'' refer to constant and diminishing step-sizes respectively.}
    \label{fig: Binary classification example}
\end{figure}

We consider a bilinear logistic regression model, which can be viewed as a simplified one-hidden-layer neural network for binary classification, serving as an example of non-convex distributed optimization. Given a dataset $(\xi, \zeta)$  where  $\xi = \{\xi^{(1)}, \xi^{(2)}, \dots, \xi^{(m)}\}$  are the input features and  $\zeta = \{\zeta^{(1)}, \zeta^{(2)}, \dots, \zeta^{(m)}\}$  are the corresponding binary class labels, where $\xi^{(i)} \in \mathbb{R}^{n_{2}}$ and $\zeta^{(i)} \in \{-1,1\}$ for all $i = 1,\dots,m$. As an instance of Problem \eqref{eq: Unconstrained optimization problem}, the learning problem with regularization term can be formulated as
\begin{gather*}
    \min_{\mathbf{x},\mathbf{X}}~L(\mathbf{x},\mathbf{X}), \text{ where } L(\mathbf{x},\mathbf{X}) = \sum_{i=1}^{m} L_{i}(\mathbf{x},\mathbf{X}),
\end{gather*}
with weights $\mathbf{x} \in \mathbb{R}^{n_{1}}$, $\mathbf{X} \in \mathbb{R}^{n_{1}\times n_{2}}$ of appropriate dimensions, where $L_{i}(\mathbf{x},\mathbf{X}) = \frac{1}{m}  \ln(1 + \exp(-\zeta^{(i)}\mathbf{x}^{\top}\mathbf{X}\xi^{(i)})) + \frac{\eta}{2m}(\|\mathbf{x}\|^{2} + \|\mathbf{X}\|_{F}^{2})$.

For a straightforward visualization of the learning problem, we select parameters that render all variables as scalars, setting $n_{1} = n_{2} = 1$. Consider that we engage $m=20$  agents to collaboratively solve this problem, connecting them in a regular graph with maximum degree 19. Note that each agent $i$ only knows its local objective $L_{i}(\mathbf{x},\mathbf{X})$. Fig. \eqref{fig: Contour plot} shows the saddle point (marked by a red $\times$) located at $(0,0)$, with the two local minimizers (marked by a red $\star$) symmetrically positioned in the positive and negative quadrants. In the following instance, we set $\eta = 0.1$ and generate $\zeta^{(i)}$ uniformly from $\{-1,1\}$ and $\xi^{(i)} \sim \mathcal{N}(\zeta^{(i)},1)$. We generate mixing matrices satisfying Assumption \ref{ass: Mixing matrix} and set the constant step-size  $\alpha = 0.2$ by empirical tuning. As shown in Fig.~\eqref{fig: Contour plot}, the blue trajectory starts near the stable manifold of the saddle point at $(0,0)$ (marked by a red $\times$), causing the iterates to be temporarily attracted to the saddle before converging to a local minimizer (marked by a red $\star$). In contrast, the green trajectory benefits from a better initialization and escapes the saddle region more efficiently. Fig.~\eqref{fig: Iteration plot} further illustrates that, compared to \textbf{DGD} with diminishing step-size $\alpha^{k} = \frac{2}{k+1}$ \cite{zeng2018nonconvex}, \textbf{EXTRA} with a constant step-size escapes the saddle point more effectively and converges to a local minimizer at a faster rate.

%% file: Sections/Conclusions_and_Discussion.tex
This paper advances the theoretical understanding of \textbf{EXTRA} in non-convex distributed optimization by modeling it as a dynamical system. We show that \textbf{EXTRA} converges to a consensual first-order stationary point and almost surely avoids strict saddle points, offering second-order guarantees.